\documentclass[12pt]{amsart}
\usepackage[dvipdfmx]{graphicx}
\usepackage{color}
\usepackage{amsmath,amssymb,amsthm}
\usepackage{amscd}
\usepackage{comment}
\usepackage{overpic}
\usepackage{tabularx}
\usepackage[dvipdfm,left=30truemm,right=30truemm,top=30truemm,bottom=30truemm]{geometry}

\usepackage{bm}

\theoremstyle{plain} 
\newtheorem{theorem}{Theorem}[section]

\theoremstyle{definition}

\makeatletter
\def\Bline{%
\noalign{\ifnum0=`}\fi\hrule \@height 1pt \futurelet
\reserved@a\@xhline}
\@addtoreset{equation}{section}

\makeatother

\allowdisplaybreaks[3]

\newcommand{\R}{\mathbb{R}}

\newcommand{\n}{\mathbf{n}}
\renewcommand{\v}{\mathbf{v}}
\newcommand{\x}{\mathbf{x}}
\newcommand{\p}{\mathbf{p}}

\renewcommand{\phi}{\varphi}
\renewcommand{\epsilon}{\varepsilon}

\newcommand{\rank}{\mathrm{rank}}

\renewcommand{\geq}{\geqslant}

\begin{document}

\title[Lagrange functions and Lagrange multipliers]
{Degeneracies of Lagrange functions and geometric interpretations of Lagrange multipliers}

\author[M.~Hasegawa]
{Masaru HASEGAWA}

\address[Masaru Hasegawa]{%
Department of Information Science, Center for Liberal Arts and Sciences, Iwate Medical University,
1-1-1 Idaidori, Yahaba-cho, Shiwa-gun, Iwate 028-3694, Japan.}
\email{mhase@iwate-med.ac.jp}



\thanks{The research underlying this work was supported by FAPESP
post-doctoral grant 2013/02543-1 during the author's post-doctoral
period at ICMC-USP}


\allowdisplaybreaks[4]



\maketitle


\begin{abstract}
We introduce Lagrange functions associated with height and distance-squared functions on implicit plane curves. 
We show that degeneracies of these functions characterize geometric properties of the curves, such as inflections and vertices, and give geometric interpretations of the corresponding Lagrange multipliers in terms of contact with lines and circles.
\end{abstract}


\section{Introduction}

Let $g:\R^2\to\R$ be a smooth function, and set
\begin{equation*}
F:\R^2\times\R\to\R,\quad F(\x,\lambda)=g(\x)+\lambda f(\x).
\end{equation*}
It is well known that $F$ is the Lagrange function and $\lambda$ is
called the Lagrange multiplier. 
The condition $\nabla F(\x,\lambda) = 0$ holds if and only if $\nabla g(\x)=-\lambda\nabla f(\x)$ and $f(\x)=0$. 
Therefore, if $\nabla f(\x)\ne 0$ and $\nabla g(\x) \ne 0$ then $\nabla F(\x,\lambda) = 0$ if and only if $f(\x)=0$ is tangent to $g(\x)=d$ at $\x$ for some $d$. 
The relation between contact geometry and singularity theory has been extensively studied (see \cite{IRRT2015} for example).
This suggests that more degenerate differential conditions on $F$ correspond to more degenerate contact between $f(\x)=0$ and $g(\x)=d$.  

In \cite{Hasegawa2017}, the author gave an answer to this question. 
For an implicit surface $M=\{\x\in\R^3|f(\x)=0\}$, where $f:\R^3\to\R$ is a smooth function, the author introduced two families of functions:
\begin{align*}
& H:\R^3\times\R\times S^2 \to \R,\quad H(\x,\lambda,\v)= \langle \x,\v \rangle+\lambda f(\x);\\
& D:\R^3\times\R\times\R^3\setminus M \to\R,\quad D(\x,\lambda,\p) = |\x-\p|^2+\lambda f(\x). 
\end{align*}
The functions $H$ and $D$ measure the contact of $M$ with planes and spheres, respectively, and hence play the roles of families of height and distance-squared functions on $M$. 

Formulas for geometric invariants of implicit curves are well known (see \cite{Goldman2005} for example), while inflections and vertices of parametrized plane curves have been studied from the viewpoint of singularity theory (see \cite{DT2018,ST2017} for example).
It is natural to consider similar families of functions on implicit plane curves. 
We investigate these invariants from a different viewpoint:
the degeneracies of the associated Lagrange functions and the geometric
meaning of the corresponding Lagrange multipliers. 
We show that these degeneracies characterize geometric properties of implicit plane curves, such as inflections and vertices, through their contact with lines and circles.

In Section 2, we introduce families of height and distance-squared functions on implicit plane curves and characterize their singularities in terms of inflections and vertices. 
In Section 3, we apply the distance-squared family to the evolute of an
implicit plane curve and show, through the ordinary $3/2$-cusp, a
difference between the evolutes obtained from the implicit and
parametric representations of the curve.

The author would like to express his sincere gratitude to Toshizumi Fukui for fruitful suggestions. 


\section{Families of height and distance-squared functions on implicit plane curves}

Let $f:\R^2\to\R$ be a smooth function, and let $C$ denote an implicit plane curve $\{(x,y)=\x\in\R^2|f(\x)=0\}$. The unit normal vector to $C$ is given by $\n=\nabla f/|\nabla f|$, where $\nabla f$ is the gradient as row vector. The curvature of $C$ is given by
\begin{equation}
\label{eq:kappa}
\kappa = -\dfrac{f_x^2 f_{yy} - 2 f_x f_y f_{xy} + f_y^2 f_{xx}}{|\nabla f|^3} = \dfrac{\begin{vmatrix}
\mathcal{H}(f) & (\nabla f)^T \\ 
\nabla f & 0
\end{vmatrix}}
{|\nabla f|^3},
\end{equation}
where $\mathcal{H}(f)$ is the Hessian matrix of $f$ (see \cite{Goldman2005} for example).

Let $g:\R^2\to\R$ be a smooth function, and set
\begin{equation}
\label{eq:Lagrange}
F:\R^2\times\R\to\R,\quad F(\x,\lambda)=g(\x)+\lambda f(\x).
\end{equation}
As mentioned above, if $\nabla f(\x)\ne 0$ and $\nabla g(\x) \ne 0$ then $\nabla F(\x,\lambda) = 0$ if and only if $C$ is tangent to $g(\x)=d$ at $\x$ for some $d$. 
It follows from that, when we take $\langle \x, \v \rangle$ $(\v\in S^1)$ and $|\x-\p|^2$ $(\p\in\R^2)$ as $g$ in \eqref{eq:Lagrange}, the condition $\nabla F=0$ holds if and only if $C$ is tangent to a line normal to $\v$ and to a circle centered at $\p$, respectively.

Consider
\begin{align*}
&H:\R^2\times\R\times S^1 \to \R,\quad H(\x,\lambda,\v) = \langle \x,\v \rangle + \lambda f(\x),\\
&D:\R^2\times\R\times \R^2 \to \R,\quad D(\x,\lambda,\p)=|\x-\p|^2 + \lambda f(\x),
\end{align*}
and define
\begin{align*}
h_{\v}(\x,\lambda)=H(\x,\lambda,\v),\quad
d_{\p}(\x,\lambda)=D(\x,\lambda,\p).
\end{align*}
According to the above reasons, we expect that $h_{\v}$ and $d_{\p}$ play the roles of the height and distance-squared functions on $C$, respectively. 

\subsection{Families of height functions}

\begin{theorem}
\label{thm:height}
Suppose that $f(\x_0)=0$ and $\nabla f(\x_0)\ne(0,0)$. Then
\begin{enumerate}
\item 
$\nabla h_{\v}=0$ at $(\x_0,\lambda_0)$ if and only if $\lambda_0=\pm 1/|\nabla f(\x_0)|$ and $\v=\mp \n(\x_0)$,  
\item 
$\nabla h_{\v}=0$ and $\det(\mathcal{H}(h_{\v}))=0$ at $(\x_0,\lambda_0)$ if and only if $\lambda_0=\pm 1/|\nabla f(\x_0)|$, $\v=\mp \n(\x_0)$ and $C$ has an inflection at $\x_0$.
 \end{enumerate}
\end{theorem}
\begin{proof}
We have $\nabla h_\v = (\v + \lambda \nabla f, f)$.
The assertion (1) is immediately deduced from this.
 
The Hessian matrix of $h_\v$ is given by
\[
\mathcal{H}(h_\v)
= 
\begin{pmatrix}
\lambda \mathcal{H}(f) & (\nabla f)^\mathrm{T} \\
\nabla f & 0
\end{pmatrix}.
\]
It follows from \eqref{eq:kappa} that
$$
\det(\mathcal{H}(h_\v)) = \lambda \kappa|\nabla f|^3 
$$
which establishes the assertion (2).
\end{proof}
Theorem \ref{thm:height} shows that $h_{\v}$ plays the role of the height function on $C$ along $\v$.

\subsection{Families of distance-squared functions}

\begin{theorem}
\label{thm:distance}
Suppose that $f(\x_0)=0$ and $\nabla f(\x_0)\ne(0,0)$. Then
\begin{enumerate}
\item
$\nabla d_{\p}=0$ at $(\x_0,\lambda_0)$ if and only if $\p=\x_0+\lambda_0|\nabla f(\x_0)|\n(\x_0)/2$ holds.
\item
$\nabla d_{\p}=0$ and $\det(\mathcal{H}(d_{\p}))=0$ at $(\x_0,\lambda_0)$ if and only if $\lambda_0=2/(\kappa(\x_0)|\nabla f(\bm{x_0})|)$ and $\p=\x_0+\n(\x_0)/\kappa(\x_0)$. 
\item 
$\nabla d_{\p}=0$, $\det(\mathcal{H}(d_{\p}))=0$ and $\rank(J(\phi_d))<3$ at $(\x_0,\lambda_0)$ if and only if $\lambda_0=2/(\kappa(\x_0)|\nabla f(\bm{x_0})|)$, $\p=\x_0+\n(\x_0)/\kappa(\x_0)$ and $C$ has a vertex at $\x_0$, where
\[
\phi_d:\R^2\times\R\to\R^4,\quad \phi_d(\x,\lambda) =(\nabla d_{\p}(\x,\lambda),\det(\mathcal{H}(d_{\p}(\x,\lambda)))) 
\]
and $J(\phi_d)$ is the Jacobian matrix of $\phi_{d}$. 
\end{enumerate}
\end{theorem}

\begin{proof}
Since we have $\nabla d_\p=(2(\x-\p)+\lambda \nabla f,f)$, we obtain 
\begin{align*}
2(\x-\p)+\lambda |\nabla f|\,\n& =0, \\
\p&=\x+\dfrac{\lambda}2|\nabla f|\,\n,
\end{align*} 
which proves the assertion (1). 

The Hessian matrix of $d_\p$ is given by
$$
\mathcal{H}(d_\p) = 
\begin{pmatrix}
\begin{pmatrix}2&0\\0&2\end{pmatrix}+\lambda \mathcal{H}(f) & (\nabla f)^T\\
\nabla f & 0
\end{pmatrix},
$$
and thus we have
$$
\det(\mathcal{H}(d_\p))) = -2|\nabla f|^2+\lambda\left|\begin{matrix}\mathcal{H}(f)&(\nabla f)^T\\\nabla f&0\end{matrix}\right|.
$$
It follows from \eqref{eq:kappa} that 
$$
\det(\mathcal{H}(d_\p))=-2|\nabla f|^2+\lambda \kappa |\nabla f|^3.
$$
Therefore, 
\begin{align*}
-2|\nabla f|^2+\lambda \kappa |\nabla f|^3 & = 0 \\
\lambda & = \dfrac2{\kappa |\nabla f|},
\end{align*}
which proves the assertion (2). 

After a suitable rotation in the $xy$-plane if necessary, we may take
the $y$-axis as the normal direction to $C$ at $\x_0$. 
So we may assume that $f_x(\x_0)=0$ and $f_y(\x_0)\ne0$.  
Let $\lambda_0=2/(\kappa(\x_0)|\nabla f(\bm{x_0})|)$ and $\p=\x_0+\n(\x_0)/\kappa(\x_0)$ hold. 
Since $f_y(\x_0)\ne0$, by the implicit function theorem, the curve $f(x,y)=0$ can be expressed locally as the graph $y=y(x)$ satisfying $f(x,y(x))=0$ close to $\x_0$. 
It follows from \eqref{eq:kappa} that 
\begin{equation*}
\dfrac{d\kappa}{dx}(\x_0)=\dfrac{f_y(\x_0)^3(-f_{xxx}(\x_0)f_y(\x_0)+3f_{xx}(\x_0)f_{xy}(\x_0))}{|\nabla f(\x_0)|^5}.
\end{equation*}
Hence, $C$ has a vertex at $\x_0$ if and only if 
\begin{equation}
\label{eq:vertex}
-f_{xxx}(\x_0)f_y(\x_0)+3f_{xx}(\x_0)f_{xy}(\x_0)=0.
\end{equation}
We shall also calculate $J(\phi_d)$. 
A direct calculation of $J(\phi_d)$, similar to that for ridge points
of implicit surfaces in \cite{Hasegawa2017}, shows that $\rank(J(\phi_d))<3$ if and only if \eqref{eq:vertex} holds. 
\end{proof}

Theorem \ref{thm:distance} shows that $d_\p$ plays the role of the distance-squared function on $C$ from $\p$. 

\section{An application to evolutes}

Theorem \ref{thm:distance} shows that the following set $\mathcal{B}(D)$ coincides with the evolute of $C$:
\[
 \mathcal{B}(D)=\{\p\in\R^2|\nabla d_{\p}(\x,\lambda,\p)=0, \det(\mathcal{H}_{d_{\p}})(\x, \lambda,\p)=0\text{ for some }(\x,\lambda,\p)\in\R^2\times \R\times \R^2\}.
\]

Let $\gamma=\gamma(t):I\subset \R \to \R^2$ be a plane curve (possibly with singularities), and let $\hat D:I\times \R^2 \to \R$ be a family of distance-squared functions on $\gamma$ defined by $\hat D(t,\ \p)=|\gamma(t)-\p|^2$. 
The distance-squared function $\hat d_\p(t) = \hat D(t,\p)$ on $\gamma$ from any point $\p$ on the normal line at the singular point $\gamma(t_0)$ has an $A_{\geq 2}$-singularity at $t_0$ (i.e., $\hat d_\p(t_0)=\hat d_\p(t_0)=0$). 
Therefore, this normal line should be regarded as part of the evolute of $\gamma$. 
The closure of the centers of osculating circles of $\gamma$ at its regular points is called the \textit{proper evolute} of $\gamma$, and the \textit{full evolute} is defined as the bifurcation set $\mathcal{B}(\hat D)$ of $\hat D$, where 
$$
\mathcal{B}(\hat D)=\left\{\p\in\R^2\left|\dfrac{d \hat D}{dt}(t,\p)=\dfrac{d^2 \hat D}{dt^2}(t,\p)=0\text{ for some }(t,\p)\in I \times \R^2\right.\right\}
$$ 
The full evolute is the union of the proper evolute and the normal line at the singular point counted with multiplicity. 
In \cite{ST2017}, the authors studied deformations of full evolutes of plane curves with singularities.   

The ordinary $3/2$-cusp $C$ has two expressions: 
one is the parameterized curve $\gamma(t)=(t^2,t^3)$; 
another is the implicit curve $f(x,y)=x^3-y^2=0$. 
The full evolute of $\gamma$ consists of the proper evolute
\begin{equation}
\label{eq:evo}
e(\gamma)(t) = \left(-\dfrac{9t^4+2t^2}2,\,\dfrac{4(3t^3+t)}3\right)
\end{equation}
and the line $x=0$ (figure \ref{fig:evo}). 
On the other hand, $\mathcal{B}(D)$ of $f=0$ consists only of the curve 
\begin{equation*}
t\mapsto \left(-\dfrac{9t^2-2t}2,\,\pm\dfrac{4\sqrt{x}(3x+1)}3\right)
\end{equation*}
whose image coincides with the proper evolute \eqref{eq:evo}. 
This example shows that $\mathcal{B}(D)$ for an implicit representation may capture only the proper evolute. 

\begin{figure}[htbp]
\centering
\includegraphics
[width=0.25\textwidth,clip]{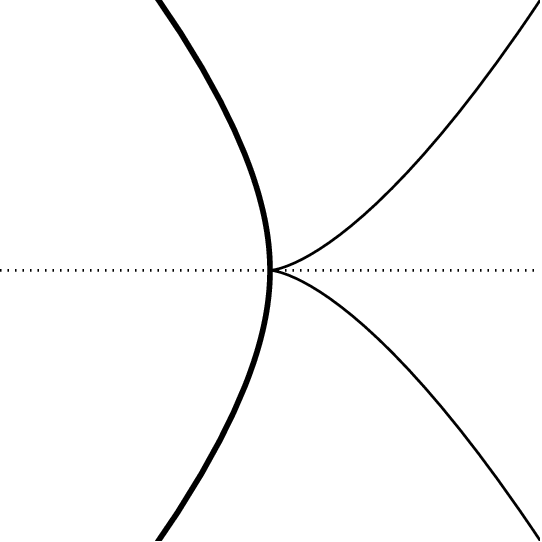}
\caption{Ordinary $3/2$-cusp and its full evolute.}
\label{fig:evo}
\end{figure}


%


\end{document}